\documentclass[12pt, reqno]{amsart}
\usepackage{amsmath, amsthm, amscd, amsfonts, amssymb, graphicx, color}

\newtheorem{theorem}{Theorem}[section]
\newtheorem{lemma}[theorem]{Lemma}
\newtheorem{proposition}[theorem]{Proposition}
\newtheorem{corollary}[theorem]{Corollary}
\theoremstyle{definition}
\newtheorem{definition}[theorem]{Definition}
\newtheorem{example}[theorem]{Example}

\theoremstyle{remark}
\newtheorem{remark}[theorem]{Remark}
\numberwithin{equation}{section}

\begin{document}
\setcounter{page}{1}

\title[Element absorb topology on rings]{\textbf{Element absorb topology on rings}}

\author[ Ali Shahidikia]{ Ali Shahidikia }
\address{Department of Mathematics, Dezful branch, Islamic Azad University, Dezful, Iran}
\email{\textcolor[rgb]{0.00,0.00,0.84}{ali.Shahidikia@iaud.ac.ir}}


\begin{abstract}
In this paper, we introduce a new Topology related to special elements in a noncummutative rings. Consider a ring $R$, we denote by
$\textrm{Id}(R)$ the set of all idempotent elements in $R$.  Let $a$ is an element  of $R$. The element absorb Topology related to $a$ is defined as $\tau_a:=\{ I\subseteq R | Ia \subseteq I\} \subseteq P(R)$. Since this topology is obtained
from act of ring, it explains Some of algebraic properties of ring in Topological language .In a special case when $e$ ia an idempotent element, 
$\tau_e:=\{ I\subseteq R | Ie \subseteq I\} \subseteq P(R)$. We present Topological description of the pierce decomposition  $ R=Re\oplus R(1-e)$.  
\end{abstract}

\maketitle

\section{Introduction}
Throughout this article, all rings are associative with identity. and subrings have the same identity as the overring,
unless indicated otherwise; $R$ denotes such a ring. For any  projection invariant left ideal $Y$ (i.e.  $ Ye \subseteq Y$  for all $e=e^2$ ,$e \in R$) (see \cite{B1, B2,B3} for more details).  We motivated by this concept to define a Topology on a ring  such that  $Y \in \tau_e$.

\section{Basic Results}

In this section, we introduce the projection Invariant  Topology concept, develop fundamental results and make connections with related  algebric notions and topologic notions.
\begin{definition}
	Let $R$ be a ring and $a\in R$. The element absorb Topology related to $a$ is defined by
 $\tau_a:=\{ I\subseteq R | Ia \subseteq I\} \subseteq P(R)$.
\end{definition}
\begin{proposition}
	$(R,\tau_a)$ is a topological space. 
\end{proposition}
\begin{proof}
	\begin{enumerate}
		\item
		It is obvious that $R, \phi \in \tau_a$
		\item
		For all $\Omega_\alpha \in \tau_a$, $\alpha \in I$, we have $\cup \Omega_\alpha \in \tau_a$ since 
		$\Omega_\alpha \in \tau_a$. As $\Omega_\alpha a \subseteq\Omega_\alpha $, then
		$(\cup U_\alpha)a=\cup(U_\alpha a)\subseteq \cup U_\alpha $.
		\item
		For all $i=1,2,\ldots, n$, $\Omega_i \in \tau_a$, we have $\cap \Omega _i \in \tau_e$ since
		$\Omega_ia \subseteq \Omega_i $. As $\Omega_i \in \tau_a$, then $( \cap \Omega_i) a \subseteq \cap(\Omega_i a)\subseteq \cap \Omega_i$.
	\end{enumerate}
\end{proof}
It is trivial that every ideal and right ideal is an open subset of $(R,\tau_a)$.
\begin{definition}
	Let $R$ be a ring, $A$ be a subset of $R$, and $x\in R$.  $x$ is called  a \textit{limited point} of  $A$,  if for all $I \in \tau_a$ contain  $x$, there exists  $y \in A\setminus \{x\}\}$ such that $ya^n \in I$ .
\end{definition}
\begin{remark}
	In $(R,\tau_a)$, any union (intersection) of closed (open) sets is closed (open) .
\end{remark}

For all $A \subseteq R$ let $\hat{A}:=\{x \in R | xa \in A\}$. So for $I \in \tau_a$ we have $I \subseteq \hat{I}$ and for $F^c \in \tau_a$, $\hat{F } \subseteq F$.
\begin{theorem}
	The function $f:(R,\tau_a) \rightarrow (R,\tau_a) $ with  $ f(x)=xa $   is continious.   
\end{theorem}
 \begin{proof}
	Let $B$ be an open set contain $f(x)$, then  $xa \in B$, so $x \in \hat{B}$ and $f(x) \in f( \hat{B})=\hat{B}a \subseteq B$.  
\end{proof}
\begin{definition}
	Consider the Topological space $(R,\tau_a)$. we denote by $\mathrm{O}(x)$
	the set	 $\{ x,xa, \ldots \}$. 
\end{definition}
 It is obvious that $\mathrm{O}(x) \in \tau_a$.
This means that for any $x \in R$, $\mathrm{O}(x)$ is an open set
including $x$. Also $\mathrm{O}(x)$ is included in any open set including $x$. For any subset $A$ of $R$, $\mathrm{O}(A)=\cup_{x \in A} \mathrm{O}({x})$.

For an idempotent $e \in R  $ the pierce decomposition of $R$ related to $e$ is $R=Re\oplus R(1-e)$. If $x \in Re$, then $\mathrm{O}(x)=\{ x\}$, and if 
$x\in R(1-e)$, then $\mathrm{O}(x)=\{x,0\}$ . Since $\mathrm{O}(x)e \subseteq \mathrm{O}(x)$, then $\mathrm{O}(x) \in \tau_e$  is the smallest open set including $x$. \\

The \textit{interior} of a set  $A \subseteq R$ is the union of all open sets contained in $A$; equivalently the interior of $A$ is the largest open set contained in $A$.  We use  $\mathrm{Int}(A)$ to denote this set. Obviously, a set is open if and only if it is equal to its own interior.

\begin{remark}
	Consider the Topological space $(R,\tau_a)$. Then we have the following.
	\begin{enumerate}
		\item
		For any $a \in R$, $\mathrm{O}(x)$  is the smallest open set including $x$. This means $\mathrm{O}( x)$ is the intersection of all open sets including $x$;
		\item
		The set $\{ \mathrm{O}(x) | x \in R\}$ is a base for $(R,\tau_a)$;
		\item
		If $x \in Re$ then $\mathrm{O}(x)=\{x\}$;
		\item
		If $x \in R(1-e)$ then $\mathrm{O}(x)=\{x,0\}$.
	\end{enumerate}
\end{remark}
\begin{theorem}
	Topological space $(R,\tau_a)$ has $T_0$ property  if and only if for any $x,y\in R$,	$x\not =y$, either $x\not = ya^n  $ or $y\not = xa^n$, $n=1,2,3, \ldots$.
\end{theorem}
\begin{proof}	
	consider $x,y \in R$ with $x\not =y$. Then there	is $(G \in \tau_a)$ such that either $x \in G$, $y \in R \diagdown G$  or $y \in G$, $x \in R \diagdown G$.  So $xa^n \in G$, $y \not = xa^n$, or
	$x \in G$, $y \in R \diagdown G$. Hence $ya^n \in G$, $x \not =ya^n$.
\end{proof}
\begin{theorem}
	Topological space $(R,\tau_a)$ has $T_1$  property if and only if for any $x,y\in R$ 
	$x\not =y$,  $x\not = ya^n  $ and $y\not = xa^n$, $n=1,2,3, \ldots$.
\end{theorem}
\begin{proof}
	Consider $x,y \in R$ with $x\not =y$. By the assumption for any $n=1,2 \ldots$ $ y \not = xa^n$
	and $x \not = ya^n$. Hence $x \not \in \mathrm{O}(y)$ and $y \not \in \mathrm{O}(x)$. Now set $U = \mathrm{O}(x)$  and  $V = \mathrm{O}(y)$.
\end{proof}

\begin{theorem}	
	Topological space $(R,\tau_a)$ has $T_2$ property if and only if for any $x,y\in R$ 
	$x\not =y$, then $ya^m  \not = xa^n$, for $m,n=1,2, \ldots$.
\end{theorem}
\begin{proof}
	Suppose that $(R,\tau_a)$ has the property of $T_2$ and consider $x,y\in R$  with $x\not =y$. So there are
	open sets $U$ including $x$ and $V$ including $y$ such that $U \bigcap V =\phi $ . Hence $O(x) \cap O(y) = \phi$. Then for all $m,n=1,2, \ldots$, $ya^m  \not = xa^n$. For the converse take $\mathrm{O}(x) = U$  and  $\mathrm{O}(y) = V$ .
\end{proof}
A Topological space $(R,\tau_a)$ is called regular if for any $x \in R$ and closed set $F \subseteq R\backslash x $ with $F \not = \{O(x)\}^c$ we have $O_h(F)  \cap  O(x)  =\phi$.
\begin{theorem}
	The topological space $(R,\tau_a)$ is regular if and only if for any $x\in R$ and any open	set $U \subseteq \mathrm{O}(x)$ including x, we have $x \in \mathrm{O}(x) \subseteq \overline{\mathrm{O}(x)}\subseteq  U$.
\end{theorem}

\begin{proof}
	Suppose that $(R,\tau_a)$ is regular and $x \in U\in \tau_a$. Then
	\[\mathrm{O}_h(R \diagdown U) \cap \mathrm{O}(x) = \phi  \quad \quad(1)\] 
	On the other hand always
	$R \diagdown U  \subseteq \mathrm{O}(R \diagdown U) $,  so
	\[[\mathrm{O}_h(R \diagdown U)]^c \subseteq  U \quad \quad (2) \]
	
	Now from (1) and (2) we have $x \in \mathrm{O}(x) \subseteq \mathrm{O}_h[R \diagdown U)]^c\subseteq  U$. Since $\mathrm{O}_h[R \diagdown U)]^c$ is closed then
	$ \overline{\mathrm{O}(x)}\subseteq \mathrm{O}_h( [R \diagdown U]^c)$. Hence $x \in \mathrm{O}(x) \subseteq \overline{\mathrm{O}(x)}\subseteq \mathrm{O}_h( [R \diagdown U]^c)$.
	For the converse consider an element $x \in R$ and a closed set $F$ with $\mathrm{O}(x)^c \not = F\subseteq R\diagdown x $. These imply that
	$R\diagdown  F \not= \mathrm{O}(x)$ and $x \in R\diagdown F \in \tau_a$.  According to the assumption $x \in \mathrm{O}(x) \subseteq \overline{\mathrm{O}(x)}$ So
	$x \in \mathrm{O}(x)$ and $F \subseteq R\diagdown \overline{\mathrm{O}(x)}$ which imply that $\mathrm{O}_h(F) \subseteq R\diagdown \overline{\mathrm{O}(x)}$ . Since 
	$R\diagdown \overline{\mathrm{O}(x)} \cap \overline{\mathrm{O}(x)}= \phi$. Hence $\mathrm{O}_h(F) \cap \mathrm{O}(x)=\phi$.
\end{proof} 
In general topology a family $\{ G_\alpha : \alpha \in \sqcup \}$ of open sets in $R$ is called an (open) cover of $I$ if
$I \subseteq \cup G_\alpha$

We recall that  $a\in R$ is left (right) S-unital if there exists $x \in R$ such that $ax=x (xa=x)$.  
\begin{theorem}
	Let$(R, \tau_a)$ be a hausdorff space.  Then $a \in R$ is left S-unitral if and only if for any  open cover
	$\{ G_\alpha : \alpha \in \sqcup \}$ of $R$ there exist $x_0 \in R, \alpha_0 \in \sqcup  $ such that $x_0, x_0a \in G_{\alpha_0}$.
\end{theorem}

\begin{proof}
	If $xa=x$, then it is trivial that $\mathrm{O}(x)=\{x,xa,xa^2, \ldots \}=\{x\}$ contains $x, xa$. 
	For the converse suppose that $xa\not =x $ for some $x\in R$. Since $(R, \tau_a)$
	is a hausdorff space, then for  $x \in R$ there are  $U_x$ including $x$ and $W_x$  including $xa$ such that
	$U_x \cap W_x = \phi $. Also, according to the Proposition 2.5, we can choose $U_x$ such that $U_x a \subseteq W_x$. Since
	$\{U_x | x\in R \}$ is a cover for $R$,  there are $z \in R$ and $x_0 \in R$ such that both of $z$ and $za$ are
	included in $U_{x_0}$. Hence $za \in U_{x_0} a \subseteq W_{x_0}$. So  $za \in U_{x_0}  \cap W_{x_0}$,
	which implies that $U_{x_0}  \cap W_{x_0} = \phi$.
\end{proof}
\begin{lemma}
	For any $x \in R$, $\overline{\mathrm{O}(x)}$ is  closed and open.
\end{lemma}
\begin{proof}
	It is clear that $\overline{\mathrm{O}(x)}$  is closed. Suppose that $W$ is an arbitrary open set including $ya$. So
	there is $U$ such that  $Ua \subseteq W$. Now let $y$ be a cluster point of $\mathrm{O}(x)$. So  $U \cap \mathrm{O}(x) \not = \phi $.
	Hence  $Ua \cap \mathrm{O}(x) = \phi $, which implies that  $W \cap \mathrm{O}(x) \not = \phi $. This guarantees that  $ya$ is a cluster
	point in $\mathrm{O}(x)$. Hence $ya \in \overline{\mathrm{O}(x)}$   which implies that $\overline{\mathrm{O}(x)}a \subseteq \overline{\mathrm{O}(x)}$.
\end{proof}
The following Lemma will be needed in the sequel.
\begin{lemma}[\cite{Engelking}, Theorem 6.1.1]
	For every topological space $X$ the following conditions are equivalent:
	\begin{enumerate}
		\item
		The space $X$ is connected;
		\item 
		The empty set and the whole space are the only closed-and-open subsets of the space $X$;
		\item
		If  $X=X_1\cup X_2$ and the sets $X_1\quad and \quad X_2$  are separated, then one of them is empty;
		\item
		Every continuous mapping  $f:X\longrightarrow D$  of the space $X$ to the two-point discrete space
		$D = \{0,1\}$ is constant, i.e., either $f(X) \subseteq {0}  \quad or \quad   f(X) \subseteq {1}$.
	\end{enumerate}
\end{lemma}
\begin{corollary}
	$(R, \tau_a)$ is disconnected.
\end{corollary}
\begin{proof}
	Since $x \in R$, $\overline{\mathrm{O}(x)}$ is  closed-and-open. Thus  $(R, \tau_a)$ is disconnected by applying the above theorem.
\end{proof}
\begin{theorem}
	Let $(R, \tau_a)$ be a ring with local Topology.  Consider the following conditions.
	\begin{enumerate}
		\item
		For some $x_0$, the set $\mathrm{O}(x_0)$ is compact;
		\item
		If $xa \not = x$, then $x \not \in \overline{\mathrm{O}(xa^2)} $.
	\end{enumerate}
If $(R, \tau_a)$ satisfies these conditions, then there is a cluster point $y$ in $\mathrm{O}(x_0)$ such that $ya=y$.
\end{theorem}
\begin{proof}
	Consider $M=\{A \subseteq  \overline{\mathrm{O}(x_0) } | A \quad and \quad \bar{A} \in \tau_a \}$.  
	Lemma 2.13 implies that  $M$ is nonempty. Let $M$ be partially ordered by the set inclusion and let $N$  be a chain in $M$. Put $M_0= \cap\{A | A \in N \}  $. $M_0$ is closed nonempty subset of $\mathrm{O}(x_0)$ by the compactness of
	$\overline{\mathrm{O}(x_0)} $ and it is a lower bound of $N$. 
	Using Zorn's Lemma we can find a subset  $L$ of $M$ which is
	minimal with respect to being nonempty, closed and mapped into itself by $f$. By the minimality of $L$ we have $T( L) =  L$.
	Let $x$ be an element in  $L$ and suppose that $x \not = xa$. Then $x \not \in \mathrm{O}(xa^2)$ and so the continuity
	of f implies that the set $\mathrm{O}(xa^2)$ is mapped into itself by $f$ and the minimality of  $L$ implies that
	$L=\overline{\mathrm{O}(xa^2)}$. On the other hand we have $x \in  L$. It follows that $x  \in \overline{\mathrm{O}(xa^2)} $ which is desire contradiction. Therefore $xa=x$.
\end{proof}
\section{Projection Invariant  Topology }
In following we focus on Projection Invariant  Topology and idempotents of rings.

\begin{theorem}
	Let $(R, \tau_e)$ be a ring with projection invariant Topology. Then $Re$ is an open and dense left ideal. 
\end{theorem}
\begin{proof}
	let $x \in Re$. Then $x=re, r\in R$ and $xe=x$. So we have $\mathrm{O}(x)=\{x\} \in \tau_e$. Then each element is interior and that  $Re$ is open.   
	Now we show that       $\overline{Re} = R$. If $x \in Re$, then $\mathrm{O}(x) \cap Re\setminus x = \phi$. Then there is no cluster point in $Re$. Thus
	$(Re)' \subseteq R \setminus Re$.	
	If $x \in R\setminus Re $, then $x \not = xe$, so for every open set $U$ including $x$, $U \cap Re \setminus x \not = \phi$, then $x \in (Re)'$,
	so $ R\setminus Re \subseteq (Re)'$.
	$(Re)'=R\setminus Re$,	then $\overline{Re} = R$.	
\end{proof}

\bibliographystyle{amsplain}
\begin{theorem}
	In the Topological space $(R, \tau_e)$, $R(1-e)$  is an open and not closed left ideal.
\end{theorem}
\begin{proof}
	For all $x \in R(1-e)$, we have  open set $U= \{x,0\}$, including $x$, so $R(1-e)$ is open.
	We will show that $[R(1-e)]'=[R\setminus Re] \cap [R\setminus R(1-e)]$.
	If $x \in Re$, then $x=xe$, so   $U= \{x\}$ is an open set and $U \cap R(1-e) = \phi$. Then
	$[R(1-e)]' \subseteq [R\setminus Re] \cap [R\setminus R(1-e)]$. For the converse if $x \in[R\setminus Re] \cap [R\setminus R(1-e)] $, then
	we have $xe \not = x, xe \not = 0$, so $U= \{x,xe\} \cap R(1-e) \setminus x \not = \phi $, so $x \in [R(1-e)]'$. Then
	$[R(1-e)]'=[R\setminus Re] \cap [R\setminus R(1-e)]$.	
\end{proof}

Therefore, in pierce decomposition $R=Re \oplus R(1-e)$ of $(R,\tau_e)$,  $Re$ is an open and dense left ideal, and  $R(1-e)$, is an open and not closed left ideal.

In the definition of Local Topology we can use  left product and the same result will be repeated. (use $a^nx$, instead of $xa^n$)

\begin{example}
	Suppose that $\tau_e=\{ I\subseteq R | eI \subseteq R \}$, and $F$ is a field. Then
	$ R= \begin{pmatrix}  
		F&F \\
		0&F
	\end{pmatrix}$, for idempotent $e^2=e$,
	$ e= \begin{pmatrix}  
		1&1 \\
		0&0
	\end{pmatrix}$,
	In $(R,\tau_e)$ we have 
	\[ R= \begin{pmatrix}  
		F&F \\
		0&F
	\end{pmatrix} =   \begin{pmatrix}  
		1&1 \\
		0&0
	\end{pmatrix} R \oplus   \begin{pmatrix}  
		0&0 \\
		0&1
	\end{pmatrix}R,\]
	now in this ring 
	$  \begin{pmatrix}  
		1&1 \\
		0&0
	\end{pmatrix}R=  \begin{pmatrix}  
		F&F \\
		0&0
	\end{pmatrix}$,
	is an open and dense ideal.
\end{example}

\end{document}